\UseRawInputEncoding
\documentclass[12pt]{article}
\usepackage[centertags]{amsmath}
\usepackage{amsfonts}
\usepackage{amssymb}
\usepackage{latexsym}
\usepackage{amsthm}
\usepackage{newlfont}
\usepackage{graphicx}
\usepackage{listings}
\usepackage{booktabs}
\usepackage{abstract}
\usepackage{enumerate}
\usepackage{xcolor}
\RequirePackage{srcltx}
\date{}

\newlength{\defbaselineskip}
\newcommand{\setlinespacing}[1]%
           {\setlength{\baselineskip}{#1 \defbaselineskip}}

\newcommand{\N}{{\mathbb{N}}}

\newcommand{\actaqed}{\hfill $\actabox$}
{\medskip\noindent \textit{Proof of #1. }}%
{\actaqed \medskip}

\def\D{{\mathcal D}}

\def\Di{{\mathcal D}}

\def\cA{{\mathcal A}}
\def\C{{\mathcal C}}
\def\cC{{\mathcal C}}

\def\cG{{\mathcal G}}

\def \Tr{\mathcal T}
\def \cK{\mathcal K}

\def \cX{\mathcal X}

\def \cM{\mathcal M}
\def\R{{\mathbb R}}
\def\Z{\mathbb Z}

\def \T{\mathbb T}

\def\bbC{\mathbb C}
\def \<{\langle}
\def\>{\rangle}

\def \Og{\Omega}

\def \va{\varepsilon}

\def \ff{\varphi}
\def\al{\alpha}
\def\bt{\beta}

\def \sp{\operatorname{span}}

\def\bx{\mathbf x}

\def\bk{\mathbf k}
\def\bj{\mathbf j}

\def\bs{\mathbf s}

\def\bN{\mathbf N}

\def\bW{\mathbf W}

\def\bF{\mathbf F}

\def\bA{\mathbf A}

\def\bt{\beta}

\newtheorem{Theorem}{Theorem}[section]
\newtheorem{Lemma}{Lemma}[section]
\newtheorem{Definition}{Definition}[section]
\newtheorem{Proposition}{Proposition}[section]

\newtheorem{Corollary}{Corollary}[section]
\numberwithin{equation}{section}

\newcommand{\be}{\begin{equation}}
\newcommand{\ee}{\end{equation}}

\def\al{{\alpha}}

\def\Bl{\Bigl}
\def\Br{\Bigr}
\def\f{\frac}

\def\vi{\varphi}

\def\va{\varepsilon}

\def\CD{{\mathcal D}}

\def\CC{{\mathbb C}}

\def\NN{{\mathbb N}}

\def\Og{\Omega}

\def\sub{\substack}

\def\Ld{\Lambda}

\def\cA{\mathcal{A}}

\def\cX{\mathcal{X}}

\def\spn{\operatorname{span}}
\def\bx{\mathbf{x}}

\def\bx{\mathbf{x}}

\DeclareFontEncoding{FMS}{}{}
\DeclareFontSubstitution{FMS}{futm}{m}{n}
\DeclareFontEncoding{FMX}{}{}
\DeclareFontSubstitution{FMX}{futm}{m}{n}
\DeclareSymbolFont{fouriersymbols}{FMS}{futm}{m}{n}
\DeclareSymbolFont{fourierlargesymbols}{FMX}{futm}{m}{n}
\DeclareMathDelimiter{\VT}{\mathord}{fouriersymbols}{152}{fourierlargesymbols}{147}

\begin{document}

\title{Sparse sampling recovery in integral norms on some function classes}

\author{    V. Temlyakov 	\footnote{
		This work was supported by the Russian Science Foundation under grant no. 23-71-30001, https://rscf.ru/project/23-71-30001/, and performed at Lomonosov Moscow State University.
  }}

\newcommand{\Addresses}{{
  \bigskip
  \footnotesize

  V.N. Temlyakov, \textsc{ Steklov Mathematical Institute of Russian Academy of Sciences, Moscow, Russia;\\ Lomonosov Moscow State University; \\ Moscow Center of Fundamental and Applied Mathematics; \\ University of South Carolina, USA.
  \\
E-mail:} \texttt{temlyakovv@gmail.com}

}}
\maketitle

\begin{abstract}{This paper is a direct followup of the recent author's paper. 
In this paper we continue to analyze approximation and recovery properties with respect to systems satisfying universal sampling discretization property and a special unconditionality property. In addition we assume that the subspace spanned by our system satisfies some Nikol'skii-type inequalities. 
We concentrate on recovery with the error measured in the $L_p$ norm for $2\le p<\infty$. We apply a powerful nonlinear approximation method -- the Weak Orthogonal Matching Pursuit (WOMP) also known under the name Weak Orthogonal Greedy Algorithm (WOGA). We establish that the WOMP based on good points for the $L_2$-universal discretization provides good recovery in the $L_p$ norm for $2\le p<\infty$. For our recovery algorithms we obtain both the Lebesgue-type inequalities for individual functions and the error bounds for special classes of multivariate functions. 

We combine here two deep and powerful techniques -- Lebesgue-type inequalities for the WOMP and theory of the universal sampling dicretization -- in order to obtain new results in sampling recovery.
		}
\end{abstract}


{\it MSC classification 2000:} Primary 65J05; Secondary 42A05, 65D30, 41A63.

\section{Introduction}
\label{I}

This paper is a followup of the recent paper \cite{VT202}. We formulate here the results from \cite{VT202}, which we use in our analysis. The reader can find a discussion of the previous results directly 
connected with the results from \cite{VT202} and from this paper in \cite{VT202}. Results on discretization can be found in the recent survey papers \cite{DPTT} and \cite{KKLT}. 

 We begin with a brief 
description of  some  necessary concepts on 
sparse approximation.   Let $X$ be a Banach space with norm $\|\cdot\|:=\|\cdot\|_X$, and let $\D=\{g_i\}_{i=1}^\infty $ be a given (countable)  system of elements in $X$. Given a finite subset $J\subset \NN$, we define $V_J(\D):=\spn\{g_j:\  \ j\in J\}$. 
For a positive  integer $ v$, we denote by $\mathcal{X}_v(\D)$ the collection of all linear spaces $V_J(\D)$  with   $|J|=v$, and 
denote by $\Sigma_v(\D)$  the set of all $v$-term approximants with respect to $\D$; that is, 
$
\Sigma_v(\D):= \bigcup_{V\in\cX_v(\D)} V.
$
Given $f\in X$,  we define
$$
\sigma_v(f,\D)_X := \inf_{g\in\Sigma_v(\D)}\|f-g\|_X,\  \ v=1,2,\cdots.
$$ 
Moreover,   for a function class $\bF\subset X$, we define 
$$
 \sigma_v(\bF,\D)_X := \sup_{f\in\bF} \sigma_v(f,\D)_X,\quad  \sigma_0(\bF,\D)_X := \sup_{f\in\bF} \|f\|_X.
 $$
 
 In this paper we consider the case $X=L_p(\Omega,\mu)$, $1\le p<\infty$. More precisely, let $\Omega$ be a compact subset of $\R^d$ with the probability measure $\mu$ on it. By the $L_p$ norm, $1\le p< \infty$, of the complex valued function defined on $\Omega$,  we understand
$$
\|f\|_p:=\|f\|_{L_p(\Omega,\mu)} := \left(\int_\Omega |f|^pd\mu\right)^{1/p}.
$$
In the case $X=L_p(\Omega,\mu)$ we sometimes write for brevity $\sigma_v(\cdot,\cdot)_p$ instead of 
$\sigma_v(\cdot,\cdot)_{L_p(\Omega,\mu)}$.

We study systems, which have  the universal sampling discretization property. 
 
  \begin{Definition}\label{ID1} Let $1\le p<\infty$. We say that a set $\xi:= \{\xi^j\}_{j=1}^m \subset \Omega $ provides {\it  $L_p$-universal sampling discretization}   for the collection $\cX:= \{X(n)\}_{n=1}^k$ of finite-dimensional  linear subspaces $X(n)$ if we have
 \be\label{ud}
\frac{1}{2}\|f\|_p^p \le \frac{1}{m} \sum_{j=1}^m |f(\xi^j)|^p\le \frac{3}{2}\|f\|_p^p\quad \text{for any}\quad f\in \bigcup_{n=1}^k X(n) .
\ee
We denote by $m(\cX,p)$ the minimal $m$ such that there exists a set $\xi$ of $m$ points, which
provides  the $L_p$-universal sampling discretization (\ref{ud}) for the collection $\cX$. 

We will use a brief form $L_p$-usd for the $L_p$-universal sampling discretization (\ref{ud}).
\end{Definition}

We use $C$, $C'$ and $c$, $c'$ to denote various positive constants. Their arguments indicate the parameters, which they may depend on. Normally, these constants do not depend on a function $f$ and running parameters $m$, $v$, $u$. We use the following symbols for brevity. For two nonnegative sequences $a=\{a_n\}_{n=1}^\infty$ and $b=\{b_n\}_{n=1}^\infty$ the relation $a_n\ll b_n$ means that there is  a number $C(a,b)$ such that for all $n$ we have $a_n\le C(a,b)b_n$. Relation $a_n\gg b_n$ means that 
 $b_n\ll a_n$ and $a_n\asymp b_n$ means that $a_n\ll b_n$ and $a_n \gg b_n$. 
 For a real number $x$ denote $[x]$ the integer part of $x$, $\lceil x\rceil$ -- the smallest integer, which is 
 greater than or equal to $x$.

 Denote for a set $\xi$ of $m$ points from $\Omega$
\be\label{muxi}
\Og_m:=\xi:=\{\xi^1,\cdots, \xi^m\}\subset \Omega,\   \  \mu_m:= \frac1{m}\sum_{j=1}^m \delta_{\xi^j},\  \  \ \text{and}\  \ \mu_\xi := \f{\mu+\mu_m}2,
\ee
where $\delta_\bx$ denotes the Dirac measure supported at a point $\bx$.

 	For convenience, we will use the notation $\|\cdot\|_{L_p(\nu)}$  to denote the norm of $L_p$ defined with respect to a measure $\nu$ on $\Og$. Sometimes we use the notation $\|\cdot\|_{p}$. 
	
The main goal of this paper is to study optimal recovery on specific classes of multivariate functions. 	
 	 
 For a function class $\bF\subset \cC(\Omega)$,  we  define (see \cite{TWW})
$$
\varrho_m^o(\bF,L_p) := \inf_{\xi } \inf_{\cM} \sup_{f\in \bF}\|f-\cM(f(\xi^1),\dots,f(\xi^m))\|_p,
$$
where $\cM$ ranges over all  mappings $\cM : \bbC^m \to   L_p(\Omega,\mu)$  and
$\xi$ ranges over all subsets $\{\xi^1,\cdots,\xi^m\}$ of $m$ points in $\Og$. 
Here, we use the index {\it o} to mean optimality.  

In \cite{VT202} we studied the behavior of $\varrho_m^o(\bF,L_p)$ for a new kind of classes -- 
$\bA^r_\bt(\Psi)$. We obtained  the order of $\varrho_m^o(\bA^r_\bt(\Psi),L_p)$ (up to a logarithmic in $m$ factor) in the case $1\le p\le 2$ for the class of uniformly bounded orthonormal systems $\Psi$.
We give a definition of classes $\bA^r_\bt(\Psi)$.  For a given $1\le p\le \infty$,
let $\Psi =\{\psi_{\bk}\}_{\bk \in \Z^d}$, $\psi_\bk \in \cC(\Omega)$, $\|\psi_\bk\|_p \le B$, $\bk\in\Z^d$,  be a system in the space $L_p(\Omega,\mu)$. We consider functions representable in the form of an  absolutely convergent series
\be\label{Irepr}
f = \sum_{\bk\in\Z^d} a_\bk(f)\psi_\bk,\quad \sum_{\bk\in\Z^d} |a_\bk(f)|<\infty.
\ee
For $\bt \in (0,1]$ and $r>0$ consider the following class $\bA^r_\bt(\Psi)$ of functions $f$, which have representations (\ref{Irepr}) satisfying the following conditions
\be\label{IAr}
  \left(\sum_{[2^{j-1}]\le \|\bk\|_\infty <2^j} |a_\bk(f)|^\bt\right)^{1/\bt} \le 2^{-rj},\quad j=0,1,\dots  .
\ee

In the special case, when $\Psi$ is the trigonometric system $\Tr^d := \{e^{i(\bk,\bx)}\}_{\bk\in \Z^d}$ we proved in \cite{VT202} (see Theorem 4.6 there) the following lower bound  
\be\label{Arlb}
\varrho_m^o(\bA^r_\bt(\Tr^d),L_1) \gg m^{1/2-1/\bt-r/d}.
\ee 
 In \cite{VT202} we complemented this lower bound by the following upper bound. 
 Assume that $\Psi$ is a uniformly bounded $\|\psi_\bk\|_\infty \le B$, $\bk\in\Z^d$, orthonormal system.
Let $1\le p\le 2$, $\bt \in (0,1]$, and $r>0$, then
$$
 \varrho_{m}^{o}(\bA^r_\bt(\Psi),L_p(\Omega,\mu))\le  \varrho_{m}^{o}(\bA^r_\bt(\Psi),L_2(\Omega,\mu))
$$
\be\label{Arub}
  \ll  (m(\log(m+1))^{-5})^{1/2 -1/\bt-r/d} .  
\ee
It is noted in \cite{VT202} that $(\log(m+1))^{-5}$ in (\ref{Arub}) can be replaced by $(\log(m+1))^{-4}$. 
Bounds (\ref{Arlb}) and (\ref{Arub}) show that generally for uniformly bounded orthonormal systems $\Psi$, for instance for the trigonometric system $\Tr^d$, the gap between the upper and the lower bounds is 
in terms of an extra logarithmic in $m$ factor.  

In this paper alike the paper \cite{VT202} we study a special case, when $\Psi$   satisfies certain restrictions. We formulate these restrictions in the form of conditions imposed on $\Psi$.

{\bf Condition A.} Assume that $\Psi$ is a uniformly bounded system. Namely, we assume that $\Psi:=\{\ff_j\}_{j=1}^\infty$ is a system of uniformly bounded functions on $\Og \subset \R^d$ such that
\be \label{ub}
\sup_{\bx\in\Og} |\vi_j(\bx)|\leq 1,\   \ 1\leq j<\infty.
\ee

{\bf Condition B1.} Assume that  $\Psi $ is an orthonormal system.

{\bf Condition B2.} Assume that  $ \Psi $ is a Riesz system, i.e.
for any $N\in\N$ and any $(a_1,\cdots, a_N) \in\bbC^N,$
\begin{equation}\label{Riesz}
R_1 \left( \sum_{j=1}^N |a_j|^2\right)^{1/2} \le \left\|\sum_{j=1}^N a_j\ff_j\right\|_2 \le R_2 \left( \sum_{j=1}^N |a_j|^2\right)^{1/2},
\end{equation}
where $0< R_1 \le R_2 <\infty$.  

{\bf Condition B3.} Assume that  $\Psi$ is a Bessel system, i.e. there exists a constant $K>0$ such that for any $N\in\N$ and   for any $(a_1,\cdots, a_N) \in\bbC^N,$
\begin{equation}\label{Bessel}
  \sum_{j=1}^N |a_j|^2 \le K  \left\|\sum_{j=1}^N a_j\ff_j\right\|^2_2 .
\end{equation}

Clearly, Condition B1 implies Condition B2 with $R_1=R_2=1$. Condition B2 implies Condition B3 with 
$K=R_1^{-2}$.

In the case of $2<p<\infty$ the following upper bounds were proved in \cite{VT202} (see Theorem 4.2.  
and Remark 6.1 there).

 \begin{Theorem}[{\cite{VT202}}]\label{IT1}  Assume that $\Psi$ is a uniformly bounded Riesz system (in the space  $L_2(\Omega,\mu)$) satisfying (\ref{ub}) and \eqref{Riesz} for some constants $0<R_1\leq R_2<\infty$.
Let $2<p<\infty$, $\bt \in (0,1]$, and $r>0$. 
 There exist constants $c'=c'(r,\bt,p,R_1,R_2,d)$ and $C'=C'(r,\bt,p,d)$ such that       for any $v\in\N$ we have the bound   
 \begin{equation}\label{R4}
 \varrho_{m}^{o}(\bA^r_\bt(\Psi),L_p(\Omega,\mu)) \le C'  v^{1/2 -1/\bt -r/d}   
\end{equation}
for any $m$ satisfying 
$$
m\ge c'    v^{p/2} (\log v)^{2}.
$$
\end{Theorem}
In the proof of Theorem \ref{IT1} known results about universal $L_p$-discretization played an important 
role. We now formulate the corresponding result. 

	\begin{Theorem}[{\cite{DT}}]\label{IT2}   Assume that $\D_N=\{\ff_j\}_{j=1}^N$ is a uniformly bounded Riesz system  satisfying (\ref{ub}) and \eqref{Riesz} for some constants $0<R_1\leq R_2<\infty$.
		Let $2<p<\infty$ and let  $1\leq u\leq N$ be an integer. 		
		 Then for a large enough constant $C=C(p,R_1,R_2)$ and any $\va\in (0, 1)$,   there exist 
		$m$ points  $\xi^1,\cdots, \xi^m\in  \Og$  with 
\be\label{m}
		m\leq C\va^{-7}       u^{p/2} (\log N)^2,
\ee
			such that for any $f\in  \Sigma_u(\D_N)$, 
		\[ (1-\va) \|f\|_p^p \leq \frac   1m \sum_{j=1}^m |f(\xi^j)|^p\leq (1+\va) \|f\|_p^p. \]	
	\end{Theorem}

For a uniformly bounded Riesz system $\Psi$ Theorem \ref{IT1} gives the following upper bound 
\begin{equation}\label{R4n}
 \varrho_{m}^{o}(\bA^r_\bt(\Psi),L_p(\Omega,\mu)) \ll  \left(\frac{m}{(\log m)^2}\right)^{(2/p)(1/2 -1/\bt -r/d)} .  
\end{equation}

In this paper we improve (\ref{R4n}). For a uniformly bounded Riesz system $\Psi$ 
Corollary \ref{RC1} (see below) gives for $2\le p<\infty$
\begin{equation}\label{I12}
 \varrho_{m}^{o}(\bA^r_\bt(\Psi),L_p(\Omega,\mu)) \ll  \left(\frac{m}{(\log m)^4}\right)^{1-1/p -1/\bt -r/d} .  
\end{equation}
Moreover, the bound (\ref{I12}) is provided by the simple greedy algorithm -- Weak Orthogonal Matching Pursuit (see the definition below).  The upper bound (\ref{R9}) and the lower bound from Proposition \ref{RP1} give the following inequalities in the special case of the $d$-variate trigonometric system $\Tr^d$
\begin{equation}\label{I13}
 m^{1-1/p -1/\bt -r/d}\ll \varrho_{m}^{o}(\bA^r_\bt(\Tr^d),L_p) \ll  \left(\frac{m}{(\log m)^3}\right)^{1-1/p -1/\bt -r/d} .  
\end{equation}

Further discussion can be found in Section \ref{Di}. 

\section{Preliminaries}
\label{P}

The  Weak Orthogonal Matching Pursuit (WOMP)  is a  greedy algorithm defined with respect to a given dictionary (system) $\D=\{g\}$  in a  Hilbert space  equipped with the inner product $ \<\cdot,\cdot\>$ and the norm  $\|\cdot\|_H$.   It is also known
  under the name Weak Orthogonal Greedy Algorithm (see, for instance, \cite{VTbook}).\\

{\bf Weak Orthogonal Matching Pursuit (WOMP).} Let $\D=\{g\}$ be a system of  nonzero elements in $H$ such that $\|g\|_H\le 1$.  
 Let $\tau:=\{t_k\}_{k=1}^\infty\subset [0, 1]$ be a given  sequence of weakness parameters. 
Given  $f_0\in H$, we define  a sequence  $\{f_k\}_{k=1}^\infty\subset H$ of residuals  for $k=1,2,\cdots$  inductively  as follows: 
\begin{enumerate}[\rm (1)]
	\item 
 $ g_k\in \D$  is any  element  satisfying
$$
|\langle f_{k-1},g_k\rangle | \ge t_k
\sup_{g\in \D} |\langle f_{k-1},g\rangle |.
$$
 
\item  Let  $H_k := \sp \{g_1,\dots,g_k\}$, and define 
$G_k(\cdot, \CD)$ to be the orthogonal projection operator from $H$   onto the space $H_k$ .

\item   Define the residual after the $k$th iteration of the algorithm by
\begin{equation*}
f_k := f_0-G_k(f_0, \D).
\end{equation*}
\end{enumerate}

In the case when $t_k=1$ for $k=1,2,\dots$,   WOMP is called the Orthogonal
Matching Pursuit (OMP). In this paper we only consider the case when  $t_k=t\in (0, 1]$ for $k=1,2,\dots$.  The term {\it weak} in the definition of the WOMP means that at step (1) we do not shoot for the optimal element of the dictionary which attains the corresponding supremum. The obvious reason for this is that we do not know in general if the optimal element exists. Another practical reason is that the weaker the assumption is, the easier it is to realize in practice. Clearly, $g_k$ may not be unique.  However, all the  results formulated  below  are independent of the choice of the $g_k$.

For the sake of convenience in later applications, we will use the notation  $\text{WOMP}\bigl(\D; t\bigr)_H$ to   denote the WOMP   defined with respect to  a given weakness parameter $t\in (0, 1]$ and   a system $\CD$ in a Hilbert  space $H$.

 {\bf UP($u,D$). ($u,D$)-unconditional property.}   We say that a system $\D=\{\vi_i\}_{i\in I}$ of  elements 
 in a Hilbert space $H=(H, \|\cdot\|)$ 
  is ($u,D$)-unconditional with  constant $U>0$ for some integers $1\leq u\leq D$ if for any 
   $f=\sum_{i\in A} c_i \vi_i\in \Sigma_u(\D)$ with  $A\subset I$    and  $J\subset I\setminus A$ such that   $|A|+|J| \le D$,  we have
\be\label{UP}
\Bl\|\sum_{i\in A} c_i \vi_i\Br\|\leq U\inf_{g\in V_J(\D)}\Bl \|\sum_{i\in A} c_i \vi_i-g\Br\|,
\ee
where $ V_J(\CD):=\spn\{\vi_i:\  \ i\in J\}$.

We gave the above definition  for a countable (or finite) system $\D$. Similar definition can be given for any system $\D$ as well.

 \begin{Theorem}[{\cite[Corollary I.1]{LT}}]\label{PT1} Let $\CD$ be a dictionary in a Hilbert space $H=(H, \|\cdot\|)$ having  the property  {\bf UP($u,D$)}  with   constant $U>0$ for some  integers $1\leq u\leq D$.   Let $f_0\in H$, and let $t\in (0, 1]$ be a given weakness parameter. 
  Then there exists a positive constant $c_\ast:=c(t,U)$  depending only on $t$ and  $U$ such that  the $\text{WOMP}\bigl(\D; t\bigr)_H$ applied to  $f_0$ gives
  $$
  \left\|f_{\left \lceil{c_\ast v}\right\rceil} \right\| \le C\sigma_v(f_0,\D)_H,\  \ v=1,2,\cdots, \min\Bl\{u,  [ D/{(1+c_\ast)]}\Br\},
  $$
   where  $C>1$ is an absolute constant, and $f_k$ denotes the residual of $f_0$ after the $k$-th iteration of the algorithm.

\end{Theorem}

  We will  consider the Hilbert space $L_2(\Omega_m,\mu_m)$ 
 instead of  $L_2(\Omega,\mu)$,  
where $\Omega_m=\{\xi^\nu\}_{\nu=1}^m$ is a set of points that provides a good universal discretization,  and $\mu_m$ is the uniform probability measure on $\Og_m$, i.e., 
$\mu_m\{ \xi^\nu\} =1/m$, $\nu=1,\dots,m$.    Let $\CD_N(\Omega_m)$ denote  the restriction 
of a system  $\CD_N$ on the set  $\Omega_m$.
Theorem \ref{PT2} below  guarantees that the simple greedy algorithm WOMP gives the corresponding Lebesgue-type inequality in the norm $L_2(\Omega_m,\mu_m)$, and hence 
 provides good sparse recovery. 

\begin{Theorem}[{\cite{DTM2}}]\label{PT2}  Let  $\CD_N=\{\vi_j\}_{j=1}^N$ be  a uniformly bounded Riesz basis  in $L_2(\Og, \mu)$  satisfying  \eqref{ub} and \eqref{Riesz} for some constants $0<R_1\leq R_2<\infty$. 
Let $\Og_m=\{\xi^1,\cdots, \xi^m\}$ be a finite subset of $\Og$  that provides the $L_2$-universal discretization (\ref{ud}) for the collection 
$\cX_u(\CD_N)$ and a given integer $1\leq u\leq N$.   Given a weakness parameter $0<t\leq 1$,   there exists a constant integer  $c=c(t,R_1,R_2)\ge 1$  such that for any integer $0\leq v\leq u/(1+c)$ and any $f_0\in \cC(\Omega)$,   the  $$\text{WOMP}\Bl(\D_N(\Og_m); t\Br)_{L_2(\Omega_m,\mu_m)}$$    applied to $f_0$  gives 
\be\label{mp}
\|f_{cv}    \|_{L_2(\Omega_m,\mu_m)} \le C\sigma_v(f_0,\CD_N(\Omega_m))_{L_2(\Omega_m,\mu_m)}, 
\ee
and
\be\label{mp2}
\|f_{c v}\|_{L_2(\Omega,\mu)} \le C\sigma_v(f_0,\CD_N)_\infty,
\ee
where $C>1$ is an  absolute constant, and $f_k$ denotes the residual of $f_0$ after the $k$-th iteration of the algorithm.
 \end{Theorem}

Recall that  the modulus of smoothness of a Banach space  $X$ is defined as 
\begin{equation}\label{CG1}
\eta(w):=\eta(X,w):=\sup_{\sub{x,y\in X\\
		\|x\|= \|y\|=1}}\Bigg[\f {\|x+wy\|+\|x-wy\|}2 -1\Bigg],\  \ w>0,
\end{equation}
and that $X$ is called uniformly smooth  if  $\eta(w)/w\to 0$ when $w\to 0+$.
It is well known that the $L_p$ space with $1< p<\infty$ is a uniformly smooth Banach space with 
\be\label{CG2}
\eta(L_p,w)\le \begin{cases}(p-1)w^2/2, & 2\le p <\infty,\\   w^p/p,& 1\le p\le 2.
\end{cases}
\ee

 \section{New results on upper bounds}
\label{NU}

In addition to the formulated in the Introduction conditions on a system $\Psi$ we need one more property of the system $\Psi$.

{\bf $u$-term Nikol'skii inequality.} Let $1\le q\le p\le \infty$ and let $\Psi$ be a system from $L_p:=L_p(\Omega,\mu)$. For a natural number $u$ we say that the system $\Psi$ has 
the $u$-term Nikol'skii inequality for the pair $(q,p)$ with the constant $H$ if the following 
inequality holds 
\be\label{vNI}
\|f\|_p \le H\|f\|_q,\qquad \forall f\in \Sigma_u(\Psi).
\ee
 In such a case we write $\Psi \in NI(q,p,H,u)$. 
 
 Note, that obviously $H\ge 1$.
 
 \begin{Theorem}\label{NUT1}  Let  $\CD_N=\{\vi_j\}_{j=1}^N$ be  a uniformly bounded Riesz system  in $L_2(\Og, \mu)$  satisfying  \eqref{ub} and \eqref{Riesz} for some constants $0<R_1\leq R_2<\infty$. 
Let $\Og_m=\{\xi^1,\cdots, \xi^m\}$ be a finite subset of $\Og$  that provides  $L_2$-universal discretization for the collection 
$\cX_u(\CD_N)$,   $1\leq u\leq N$.   Assume in addition that $\D_N \in NI(2,p,H,u)$ with $p\in [2,\infty)$. Then, for a given  weakness parameter $0<t\leq 1$,   there exists a constant integer  $c=c(t,R_1,R_2)\ge 1$  such that for any integer $0\leq v\leq u/(1+c)$ and any $f_0\in \cC(\Omega)$,   the  
$$
\text{WOMP}\Bl(\D_N(\Og_m); t\Br)_{L_2(\Omega_m,\mu_m)}
$$    
applied to $f_0$  gives 
\be\label{mpn}
\|f_{cv}    \|_{L_2(\Omega_m,\mu_m)} \le C_0\sigma_v(f_0,\CD_N(\Omega_m))_{L_2(\Omega_m,\mu_m)}, 
\ee
\be\label{mp3n}
\|f_{c v}\|_{L_p(\Omega,\mu)} \le HC_1\sigma_v(f_0,\CD_N)_{L_p(\Omega,\mu_\xi)},
\ee
where $C_i$, $i=0,1$, are absolute constants,  $f_k$ denotes the residual of $f_0$ after the $k$-th iteration of the algorithm, and $\mu_\xi$ is defined in (\ref{muxi}).
 \end{Theorem}
 \begin{proof} We begin with the inequality (\ref{mpn}). It follows from the inequality (\ref{mp}) in Theorem \ref{PT2}. For completeness we give its simple proof here. 
Using (\ref{Riesz}), we obtain that  for any  sets $A\subset \{1,2,\cdots, N\}$ and  $\Lambda\subset \{1,\cdots, N\}\setminus A$ and for  any  sequences $\{x_i\}_{i\in A},$ $\{c_i\}_{i\in\Ld}\subset\CC$, 
$$
\Bl\|\sum_{i\in A}x_i\ff_i-\sum_{i\in\Lambda}c_i\ff_i\Br\|_{L_2(\Omega,\mu)}^2 \ge R_1^2  \sum_{i\in A} |x_i|^2 \ge R_2^{-2} R_1^2\Bl \|\sum_{i\in A}x_i\ff_i \Br\|_{L_2(\Omega,\mu)}^2,
$$
meaning  that the system  $\CD_N$ has the {\bf UP}$(v,N)$ property  with constant $U_1 = R_2/R_1$  in the  space $L_2(\Omega,\mu)$  for any integer $1\leq v< N$. 
It then follows from the  discretization inequalities  
(\ref{ud}) with $\mathbf {\mathcal{X}}=\mathbf {\mathcal{X}}_u(\D_N)$ that the system $\D:=\CD_N(\Omega_m)$, which is   the restriction 
of $\CD_N$ on $\Omega_m=\{\xi^1, \cdots,\xi^m\}$, has  
the  properties  {\bf UP($v,u$)} in the space  $L_2(\Omega_m,\mu_m)$  with  constant $U_2 =U_1 3^{1/2}$ for any integer $1\leq v\leq u$.
Thus,  applying Theorem \ref{PT1} to the discretized dictionary  $\CD_N(\Omega_m)$ in the Hilbert space  $L_2(\Omega_m,\mu_m)$, we conclude that   the algorithm  
$$
\text{WOMP}\bigl(\CD_N(\Omega_m); t\bigr)_{L_2(\Omega_m,\mu_m)}
$$ 
applied to  $f_0\Bl|_{\Og_m} $ gives
$$
\left\|f_{cv} \right\|_{L_2(\Omega_m,\mu_m)}\le C_0\sigma_v(f_0,\CD_N(\Omega_m))_{L_2(\Omega_m,\mu_m)}
$$
whenever $v+c v\leq u$, where  $c =c(t, R_1, R_2)\in\NN$. 
 This proves (\ref{mpn}) in Theorem \ref{NUT1}. 
 
 We now derive (\ref{mp3n}) from (\ref{mpn}).   Let $h\in \Sigma_v(\CD_N)$ be such that  
 $$
 \|f_0-h\|_{L_p(\mu_\xi)} = \sigma_v(f_0,\CD_N)_{L_p(\mu_\xi)}.
 $$ 
 We have
$$
\|f_{cv}\|_{L_p(\Omega,\mu)} \le \|h-f_0\|_{L_p(\Omega,\mu)} + \|h - G_{cv}(f_0,\CD_N(\Omega_m))\|_{L_p(\Omega,\mu)}.
$$
First, we estimate
$$
\|h-f_0\|_{L_p(\Omega,\mu)} \le 2^{1/p}\|h-f_0\|_{L_p(\mu_\xi)}= 2^{1/p}\sigma_v(f_0,\CD_N)_{L_p(\mu_\xi)}.
$$
Second, using that $h - G_{cv}(f_0,\CD_N(\Omega_m)) \in \Sigma_u(\CD_N)$, by our assumption $\D_N \in NI(2,p,H,u)$ and by discretization (\ref{ud})   we 
conclude that
$$
\|h - G_{cv}(f_0,\CD_N(\Omega_m))\|_{L_p(\Omega,\mu)} \le  H \|h- G_{cv}(f_0,\CD_N(\Omega_m))\|_{L_2(\Omega,\mu)} 
$$
\be\label{NU4}
 \le 2^{1/2}H\|h - G_{cv}(f_0,\CD_N(\Omega_m))\|_{L_2(\Omega_m,\mu_m)}.
 \ee
Next, (\ref{mpn}) implies 
$$
\|h - G_{cv}(f_0,\CD_N(\Omega_m))\|_{L_2(\Omega_m,\mu_m)} \le \|h-f_0\|_{L_2(\Omega_m,\mu_m)} +\|f_{cv}\|_{L_2(\Omega_m,\mu_m)} 
$$
$$
\le \|h-f_0\|_{L_p(\Omega_m,\mu_m)} +C_0\sigma_v(f_0,\CD_N(\Omega_m))_{L_2(\Omega_m,\mu_m)} 
$$
$$
\le 2^{1/p}(\|h-f_0\|_{L_p(\mu_\xi)} +C_0\sigma_v(f_0,\CD_N(\Omega_m))_{L_p(\mu_\xi)}) \le 
2^{1/p}(1+C_0)\sigma_v(f_0,\CD_N)_{L_p(\mu_\xi)}.
$$
This and (\ref{NU4}) imply
\be\label{NU5}
\|h - G_{cv}(f_0,\CD_N(\Omega_m))\|_{L_p(\Omega,\mu)} \le 2^{1/2+1/p}H(1+C_0)\sigma_v(f_0,\CD_N)_{L_p(\mu_\xi)}.
\ee
Finally,
$$
\|f_{cv}\|_{L_p(\Omega,\mu)} \le (1+2^{1/2+1/p}H(1+C_0))\sigma_v(f_0,\CD_N)_{L_p(\mu_\xi)},
$$
which completes the proof.

 \end{proof}

   For brevity denote $L_p(\xi) := L_p(\Omega_m,\mu_m)$, where $\Omega_m=\{\xi^\nu\}_{\nu=1}^m$  and  $\mu_m(\xi^\nu) =1/m$, $\nu=1,\dots,m$. Let 
$B_v(f,\D_N,L_p(\xi))$ denote the best $v$-term approximation of $f$ in the $L_p(\xi)$ norm with 
respect to the system $\D_N$. Note that $B_v(f,\D_N,L_p(\xi))$ may not be unique. Obviously,
\be\label{D5}
\|f-B_v(f,\D_N,L_p(\xi))\|_{L_p(\xi)} = \sigma_v(f,\D_N)_{L_p(\xi)}.
\ee

We proved in \cite{DTM3} the following theorem.

 \begin{Theorem}[{\cite{DTM3}}]\label{NUT2} Let $1\le p<\infty$ and let $m$, $v$, $N$ be given natural numbers such that $2v\le N$.  Let $\D_N\subset \C(\Og)$ be  a system of $N$ elements. Assume that  there exists a set $\xi:= \{\xi^j\}_{j=1}^m \subset \Omega $, which provides  the  one-sided $L_p$-universal discretization 
  \be\label{D6}
 \|f\|_p \le D\left(\frac{1}{m} \sum_{j=1}^m |f(\xi^j)|^p\right)^{1/p}, \quad \forall\, f\in \Sigma_{2v}(\D_N), 
\ee
  for the collection $\cX_{2v}(\D_N)$. Then for   any  function $ f \in \C(\Omega)$ we have
\be\label{D7}
  \|f-B_v(f,\D_N,L_p(\xi))\|_{L_p(\Omega,\mu)} \le 2^{1/p}(2D +1) \sigma_v(f,\D_N)_{L_p(\Og, \mu_\xi)}
 \ee
 and
 \be\label{D8}
  \|f-B_v(f,\D_N,L_p(\xi))\|_{L_p(\Omega,\mu)} \le  (2D +1) \sigma_v(f,\D_N)_\infty.
 \ee
 \end{Theorem}

 We now prove a version of Theorem \ref{NUT2}.
 
 \begin{Theorem}\label{NUT3} Let $2\le p<\infty$ and let $m$, $v$, $N$ be given natural numbers such that $2v\le N$.  Let $\D_N\subset \C(\Og)$ be  a system of $N$ elements such that $\D_N \in NI(2,p,H,2v)$. Assume that  there exists a set $\xi:= \{\xi^j\}_{j=1}^m \subset \Omega $, which provides the one-sided $L_2$-universal discretization 
  \be\label{D12}
 \|f\|_2 \le D\left(\frac{1}{m} \sum_{j=1}^m |f(\xi^j)|^2\right)^{1/2}, \quad \forall\, f\in \Sigma_{2v}(\D_N), 
\ee
  for the collection $\cX_{2v}(\D_N)$. Then for   any  function $ f \in \C(\Omega)$ we have
\be\label{D13}
  \|f-B_v(f,\D_N,L_2(\xi))\|_{L_p(\Omega,\mu)} \le 2^{1/p}(2DH +1) \sigma_v(f,\D_N)_{L_p(\Og, \mu_\xi)}.
 \ee
 and
 \be\label{D14}
  \|f-B_v(f,\D_N,L_2(\xi))\|_{L_p(\Omega,\mu)} \le (2DH +1) \sigma_v(f,\D_N)_\infty.
 \ee
 
 \end{Theorem}
\begin{proof} For brevity denote $g:= B_v(f,\D_N,L_2(\xi))$ and $h:=B_v(f,\D_N,L_p(\mu_\xi))$. Then
$$
\|f-g\|_{L_p(\mu)} \le \|f-h\|_{L_p(\mu)} + \|h-g\|_{L_p(\mu)}.
$$
First, we estimate
$$
\|f-h\|_{L_p(\mu)} \le 2^{1/p} \|f-h\|_{L_p(\mu_\xi)} = 2^{1/p} \sigma_v(f,\D_N)_{L_p(\mu_\xi)}.
$$
Second, using the obvious fact that $h-g\in \Sigma_{2v}(\D_N)$, by the assumption $\D_N \in NI(2,p,H,2v)$
we obtain
$$
\|h-g\|_{L_p(\mu)} \le H \|h-g\|_{L_2(\mu)}
$$
and continue by (\ref{D12}) and by the trivial inequality $\|f-g\|_{L_2(\xi)} \le \|f-h\|_{L_2(\xi)}$
$$
\le HD \|h-g\|_{L_2(\xi)} \le 2HD \|f-h\|_{L_2(\xi)} \le 2HD \|f-h\|_{L_p(\xi)} 
$$
$$
 \le  2^{1+1/p}HD \|f-h\|_{L_p(\mu_\xi)} = 2^{1+1/p}HD \sigma_v(f,\D_N)_{L_p(\mu_\xi)}.
$$
Combining the above inequalities, we complete the proof. 
\end{proof}

\section{New results on lower bounds}
\label{NL}

Let us discuss lower bounds for the nonlinear characteristic $\varrho_m^o(\bA^r_\bt(\Psi),L_p)$.
We will do it in the special case, when $\Psi$ is the trigonometric system $\Tr^d := \{e^{i(\bk,\bx)}\}_{\bk\in \Z^d}$. Denote for $\bN = (N_1,\dots,N_d)$, $N_j\in \N_0$, $j=1,\dots,d$,
$$
\Tr(\bN,d) := \left\{f=\sum_{\bk: |k_j|\le N_j, j=1,\dots,d} c_j e^{i(\bk,\bx)}\right\},\quad \vartheta(\bN) :=\prod_{j=1}^d (2N_j+1). 
$$
In this section $\Omega = \T^d$ and $\mu$ is the normalized Lebesgue measure on $\T^d$. 

\begin{Lemma}\label{NLL1} Let $1\le q\le p\le \infty$ and let  $\Tr(\bN,d)_q$ denote the unit $L_q$-ball of the subspace $\Tr(\bN,d)$. Then we have for $m\le \vartheta(\bN)/2$ that
$$
\varrho_m^o(\Tr(2\bN,d)_q,L_p) \ge c(d)\vartheta(\bN)^{1/q-1/p}  .
$$
\end{Lemma}

\begin{proof} Let a set $\xi\subset \T^d := [0,2\pi)^d$ of points $\xi^1,\dots,\xi^m$ be given. Consider the subspace
$$
T(\xi) := \{f\in \Tr(\bN,d):\, f(\xi^\nu) =0,\quad \nu=1,\dots,m\}.
$$
Let  $g_\xi\in T(\xi)$ and a point $\bx^*$ be such that $|g_\xi(\bx^*)|=\|g_\xi\|_\infty=1$. For the further argument we need some classical trigonometric polynomials. The univariate Fej\'er kernel of order $j - 1$:
$$
\mathcal K_{j} (x) := \sum_{|k|\le j} \bigl(1 - |k|/j\bigr) e^{ikx} 
=\frac{(\sin (jx/2))^2}{j (\sin (x/2))^2}.
$$
The Fej\'er kernel is an even nonnegative trigonometric
polynomial of order $j-1$.  It satisfies the obvious relations
\be\label{FKm}
\| \mathcal K_{j} \|_1 = 1, \qquad \| \mathcal K_{j} \|_{\infty} = j.
\ee
Let $\cK_\bj(\bx):= \prod_{i=1}^d \cK_{j_i}(x_i)$ be the $d$-variate Fej\'er kernels for $\bj = (j_1,\dots,j_d)$ and $\bx=(x_1,\dots,x_d)$.
We set 
\be\label{Rf}
f(\bx) := g_\xi(\bx)\cK_\bN(\bx-\bx^*).
\ee
 Then $f\in \Tr(2\bN,d)$, $f(\xi^\nu)=0$, $\nu=1,\dots,m$, and
\be\label{R19}
\|f\|_q \le \|g_\xi\|_\infty \|\cK_\bN\|_q \le C_1(d)\vartheta(\bN)^{1-1/q}.
\ee
At the last step we used the known bound for the $L_q$ norm of the Fej\'er kernel (see \cite{VTbookMA}, p.83, (3.2.7)). By (\ref{FKm}) we get
\be\label{R20}
|f(\bx^*)| \ge C_2(d) \vartheta(\bN).
\ee
By the Nikol'skii inequality for the $\Tr(2\bN,d)$ (see  \cite{VTbookMA}, p.90, Theorem 3.3.2) we obtain 
from (\ref{R20}) 
\be\label{R21}
\|f\|_p \ge C_3(d) \vartheta(\bN)^{1-1/p}.
\ee
 
Let $\cM$ be a mapping from $\bbC^m$ to $L_p$. Denote $g_0 := \cM(\mathbf 0)$. Then 
$$ 
\|f -g_0\|_p +\|-f - g_0\|_p \ge 2\|f\|_p.
$$
This, relations (\ref{R19}), (\ref{R21}), and the fact that both $f/\|f\|_q$ and $-f/\|f\|_q$ belong to $\Tr(2\bN,d)_q$ complete the proof of Lemma \ref{NLL1}.

\end{proof}

We now show how Theorem \ref{NUT2} can be used in proving that Theorem \ref{IT2} cannot be substantially improved in the sense of relations between $m$ and $u$. 

\begin{Proposition}\label{NLP1} Let $d\in \N$ and $p\in (2,\infty)$. There exists a positive constant $C(d,p)$
with the following property. For any 
$$
\D_N = \{e^{i(\bk,\bx)},\,\bk: |k_j|\le 2N_j, j=1,\dots,d  \}
$$ 
and $m\le \vartheta(\bN)/2$ there is no set of $m$ 
points,  which provides the one-sided $L_p(\T^d,\mu)$-universal discretization (\ref{D6}) for $\cX_{2v}(\D_N)$ with 
$$
v \ge C(d,p)(2D+1)^2 (\vartheta(\bN))^{2/p}.
$$
Here, $\mu$ is the normalized Lebesgue measure on $\T^d$. 
\end{Proposition}
\begin{proof} The proof goes by contradiction. Assume that  there exists a set $\xi:= \{\xi^j\}_{j=1}^m \subset \T^d $, which provides the one-sided $L_p$-universal discretization (\ref{D6})  for $\cX_{2v}(\D_N)$. 
Then, by Theorem \ref{NUT2} inequality (\ref{D7}) holds for any $f\in\cC(\T^d)$. We take the function $f$ from the proof of Lemma \ref{NLL1} defined in (\ref{Rf}). By the definition of $f$ we get $f(\xi^j)=0$ for all $j=1,\dots,m$. Therefore, $B_v(f,\D_N,L_p(\xi))=0$ and by (\ref{R21}) we obtain
\be\label{D9}
 \|f-B_v(f,\D_N,L_p(\xi))\|_p \ge C_3(d)\vartheta(\bN)^{1-1/p}.
\ee
On the other hand by (\ref{R19}) with $q=2$ we get
\be\label{D10}
\sum_{\bk\in \Z^d} |\hat f(\bk)| \le \vartheta(2\bN)^{1/2} \|f\|_2^{1/2} \le C_1(d) \vartheta(2\bN)^{1/2}\vartheta(\bN)^{1/2} \le C_1'(d)\vartheta(\bN). 
\ee
By (\ref{CG2}) we obtain that $\eta(L_p(\T^d,\mu_\xi),w) \le (p-1)w^2/2$. 
It is known (see, for instance, \cite{VTbook}, p.342) that for any dictionary $\D = \{g\}$, $\|g\|_X \le 1$ in a Banach  space $X$ with $\eta(X,w) \le \gamma w^q$, $1<q\le 2$, we have
 \be\label{R1}
 \sigma_v(A_1(\D),\D)_X \le C(q,\gamma)(v+1)^{1/q -1}.
 \ee
 Here
 $$
 A_1(\D) := \left\{f:\, f=\sum_{i=1}^\infty a_ig_i,\quad g_i\in \D,\quad \sum_{i=1}^\infty |a_i| \le 1 \right\}.
 $$
 
 Note that (\ref{R1}) is proved in \cite{VTbook} in the case of real Banach spaces. It is clear that we can get (\ref{R1}) for the complex $\D_N$ from the corresponding results for the real trigonometric polynomials from $\Tr(2\bN,d)$.
 
We now apply the   
inequality (\ref{R1}) and find from (\ref{D10})
\be\label{D11}
\sigma_v(f,\D_N)_{L_p(\T^d, \mu_\xi)} \le C_1(d,p)\vartheta(\bN) (v+1)^{-1/2}.
\ee
Substituting (\ref{D9}) and (\ref{D11}) into (\ref{D7}) we find
$$
v+1 \le C(d,p)(2D+1)^2 \vartheta(\bN)^{2/p}
$$
with some positive constant $C(d,p)$. We choose this constant $C(d,p)$ and complete the proof. 
\end{proof}
  
 \section{Optimal recovery on function classes}
 \label{R}
 
 The following Theorem \ref{RT1} was proved in \cite{VT202}.
 
  \begin{Theorem}[{\cite{VT202}}]\label{RT1} Let $1<p<\infty$, $r>0$, $\bt\in (0,1]$. Assume that $\|\psi_\bk\|_{L_p(\Omega,\mu)} \le B$, $\bk\in\Z^d$, with some probability measure $\mu$ on $\Omega$. Then there exist two constants $c^*=c(r,\bt,p,d)$ and $C=C(r,\bt,p,d)$ such that 
 for any $v\in \N$ there is a $J\in\N$ independent of the measure $\mu$, $2^J \le v^{c^*}$, with the property ($p^*:= \min\{p,2\}$)
 $$
 \sigma_v(\bA^r_\bt(\Psi),\Psi_J)_{L_p(\Omega,\mu)} \le CBv^{1/p^* -1/\bt-r/d}, \qquad \Psi_J := \{\psi_\bk\}_{\|\bk\|_\infty <2^J}.
 $$
 Moreover, this bound is provided by a simple greedy algorithm.
 \end{Theorem} 
 
 We now proceed to a new result on sampling recovery. In the proof of  Theorem \ref{RT3} below we need the following known result on the universal discretization.
 
  \begin{Theorem}[{\cite{DTM2}}]\label{RT2} Let $1\le p\le 2$. Assume that $ \D_N=\{\ff_j\}_{j=1}^N\subset \cC(\Og)$ is a  system  satisfying  the conditions  \eqref{ub} and   \eqref{Bessel} for some constant $K\ge 1$. Let $\xi^1,\cdots, \xi^m$ be independent 
 	random points on $\Og$  that are  identically distributed  according to  $\mu$. 
 	 Then there exist constants  $C=C(p)>1$ and $c=c(p)>0$ such that 
 	  given any   integers  $1\leq u\leq N$ and 
 	 $$
 	 m \ge  C Ku \log N\cdot (\log(2Ku ))^2\cdot (\log (2Ku )+\log\log N),
 	 $$
 	 the inequalities 
 	 \begin{equation}\label{Ex2}
 	 \frac{1}{2}\|f\|_p^p \le \frac{1}{m}\sum_{j=1}^m |f(\xi^j)|^p \le \frac{3}{2}\|f\|_p^p,\   \   \ \forall f\in  \Sigma_u(\D_N)
 	 \end{equation}
 hold with probability $\ge 1-2 \exp\Bl( -\f {cm}{Ku\log^2 (2Ku)}\Br)$.
\end{Theorem}

 \begin{Theorem}\label{RT3}  Assume that $\Psi$ is a uniformly bounded Riesz system (in the space  $L_2(\Omega,\mu)$) satisfying (\ref{ub}) and \eqref{Riesz} for some constants $0<R_1\leq R_2<\infty$.
Let $2\le p<\infty$  and $r>0$. For any $v\in\N$ denote $u := \lceil (1+c)v\rceil$, where $c$ is from Theorem \ref{NUT1}.  Assume in addition that $\Psi \in NI(2,p,H,u)$ with $p\in [2,\infty)$.
 There exist constants $c'=c'(r,\bt,p,R_1,R_2,d)$ and $C'=C'(r,\bt,p,d)$ such that   we have the bound   
 \begin{equation}\label{R4s}
 \varrho_{m}^{o}(\bA^r_\bt(\Psi),L_p(\Omega,\mu)) \le C' H v^{1/2 -1/\bt -r/d}   
\end{equation}
for any $m$ satisfying 
$$
m\ge c'    v  (\log(2v ))^4 .
$$
Moreover, this bound is provided by the WOMP. 
\end{Theorem}
\begin{proof} First, we use Theorem \ref{RT1} in the space $L_p(\Omega,\mu)$. In our case $B=1$ and $2<p<\infty$, which implies that 
$p^*=2$. Consider the system $\Psi_J := \{\psi_\bk\}_{\|\bk\|_\infty <2^J}$, $2^J \le v^{c^*}$, from Theorem 
\ref{RT1}. Note, that $J$ does not depend on $\mu$. 
 By Theorem \ref{NUT1} with $\D_N = \Psi_J$ and Theorem \ref{RT2} with $p=2$ and $K=R_1^{-2}$  there exist $m$ points  $\xi^1,\cdots, \xi^m\in  \Og$  with 
\be\label{mbound}
		m\leq C       u (\log N)^4,
\ee  
such that  for any given $f_0\in \cC(\Omega)$,  the WOMP with weakness parameter $t$ applied to $f_0$ with respect to the system  $\D_N(\Omega_m)$ in the space $L_2(\Omega_m,\mu_m)$ provides for 
any integer $1\le v \le u/(1+c)$
\be\label{R5}
	\|f_{cv} \|_{L_p(\Omega,\mu)} \le C'H \sigma_v(f_0,\CD_N)_{L_p(\Og, \mu_\xi)}.
	\ee
In order to bound the right side of (\ref{R5}) we apply Theorem \ref{RT1} in the space $L_p(\Og, \mu_\xi)$. 
For that it is sufficient to check that $\|\psi_\bk\|_{L_p(\Og, \mu_\xi)} \le 1$, $\bk\in\Z^d$. This follows from the assumption that $\Psi$ satisfies (\ref{ub}). Thus, by Theorem \ref{RT1} we obtain for $f_0 \in \bA^r_\bt(\Psi)$
\be\label{R6}
 \sigma_v(f_0,\Psi_J)_{L_p(\Og, \mu_\xi)} \le CHv^{1/2 -1/\bt-r/d}.
 \ee
  Combining (\ref{R6}), (\ref{R5}), and taking into account (\ref{mbound}) and the bound
  $$
  N\le 2^{d(J+1)} \le 2^d v^{c^*},
  $$
   we complete the proof.

\end{proof}
 
\begin{Corollary}\label{RC1}  Assume that $\Psi$ is a uniformly bounded Riesz system (in the space  $L_2(\Omega,\mu)$) satisfying (\ref{ub}) and \eqref{Riesz} for some constants $0<R_1\leq R_2<\infty$.
Let $2\le p<\infty$  and $r>0$.  
 There exist constants $c=c(r,\bt,p,R_1,R_2,d)$ and $C=C(r,\bt,p,d)$ such that   we have the bound   
 \begin{equation}\label{R4w}
 \varrho_{m}^{o}(\bA^r_\bt(\Psi),L_p(\Omega,\mu)) \le C  v^{1-1/p -1/\bt -r/d}   
\end{equation}
for any $m$ satisfying 
$$
m\ge c   v  (\log(2v ))^4 .
$$
Moreover, this bound is provided by the WOMP.
\end{Corollary}
\begin{proof} We prove that a system $\Psi$ satisfying the conditions of Corollary \ref{RC1} also satisfies the condition $\Psi \in NI(2,p,H,u)$ with $H= C(R_2) u^{1/2-1/p}$, which is required in Theorem \ref{RT3}. Then Corollary \ref{RC1} follows from Theorem \ref{RT3}.
Indeed, let
$$
f := \sum_{\bk \in Q} c_\bk \psi_\bk,\qquad |Q| \le u.
$$
Then
$$
\|f\|_\infty \le \sum_{\bk \in Q} |c_\bk| \le u^{1/2} \left(\sum_{\bk \in Q} |c_\bk|^{1/2}\right)^{1/2} \le u^{1/2} R_2 \|f\|_2.
$$
Using the following well known simple inequality $\|f\|_p \le \|f\|_2^{2/p}\|f\|_\infty^{1-2/p}$, $2\le p\le \infty$, we obtain
$$
\|f\|_p \le R_2^{1-2/p}u^{1/2-1/p} \|f\|_2.
$$
\end{proof}

We derived Theorem \ref{RT3} from Theorem \ref{NUT1}. We now formulate an analog of Theorem \ref{RT3}, which can be derived from Theorem \ref{NUT3} in the same way as Theorem \ref{RT3} has been derived from Theorem \ref{NUT1}.

\begin{Theorem}\label{RT4} Let $2\le p<\infty$ and let $m$, $v$ be given natural numbers.    Let $\Psi$ be  a uniformly bounded Bessel system satisfying (\ref{ub}) and (\ref{Bessel}) such that $\Psi \in NI(2,p,H,2v)$. 
There exist constants $c=c(r,\bt,p,K,d)$ and $C=C(r,\bt,p,d)$ such that   we have the bound   
 \begin{equation}\label{R4ss}
 \varrho_{m}^{o}(\bA^r_\bt(\Psi),L_p(\Omega,\mu)) \le CH v^{1/2 -1/\bt -r/d}   
\end{equation}
for any $m$ satisfying 
$$
m\ge c   v  (\log(2v ))^4 .
$$
 \end{Theorem}

In the special case $\Psi = \Tr^d$ is the trigonometric system, Corollary \ref{RC1} gives 
 \begin{equation}\label{R7}
 \varrho_{m}^{o}(\bA^r_\bt(\Tr),L_p(\Omega,\mu)) \ll  v^{1-1/p -1/\bt -r/d} \quad\text{for}\quad  m\gg    v  (\log(2v ))^4.
\end{equation}

It is pointed out in \cite{DTM2} that known results on the RIP for the trigonometric system can be used 
for improving results on the universal discretization in the $L_2$ norm in the case of the trigonometric system. We explain that in more detail.   
 Let $M\in \N$ and $d\in \N$. Denote $\Pi(M) := [-M,M]^d$ to be the $d$-dimensional cube. Consider the system $\Tr^d(M) := \Tr^d(\Pi(M))$ of functions $e^{i(\bk,\bx)}$, $\bk \in \Pi(M)$ defined on $\T^d =[0,2\pi)^d$. Then $\Tr^d(M)$ is an orthonormal system in $L_2(\T^d,\mu)$ with $\mu$ being the normalized Lebesgue measure on $\T^d$. The cardinality of this system is $N(M):= |\Tr^d(M)| = (2M+1)^d$. We are interested in bounds on $m(\cX_v(\Tr^d(M)),2)$ in a special case, when $M\le v^c$ with some constant $c$, which may depend on $d$. Then Theorem \ref{RT2} with $p=2$
 gives 
 \be\label{DT1}
 m(\cX_v(\Tr^d(v^c)),2) \le C(c,d) v (\log (2v))^4.
 \ee
  It is  stated in \cite{DTM2} that the known results of \cite{HR} and \cite{Bour} allow us to improve the bound (\ref{DT1}):
 \be\label{HR1}
 m(\cX_v(\Tr^d(v^c)),2) \le C(c,d) v (\log (2v))^3.
 \ee
 This, in turn, implies the following estimate
  \begin{equation}\label{R9}
 \varrho_{m}^{o}(\bA^r_\bt(\Tr^d),L_p) \ll  v^{1-1/p -1/\bt -r/d} \quad\text{for}\quad  m\gg    v  (\log(2v ))^3.
\end{equation}

We now prove that (\ref{R9}) cannot be substantially improved. We use Lemma \ref{NLL1}. 
Take $n\in\N$ and  set $N:=2^n-1$, $\bN := (N,\dots,N)$.
Then for $f\in \Tr(\bN,d)_2$ we have by the H{\"o}lder inequality with parameter $2/\bt$
$$
  \left(\sum_{\bk: \|\bk\|_\infty\le N} |\hat f(\bk)|^\bt\right)^{1/\bt} \le (2N+1)^{d(1/\bt-1/2)} \left(\sum_{\bk: \|\bk\|_\infty\le N} |\hat f(\bk)|^2\right)^{1/2}
$$
$$
  =(2N+1)^{d(1/\bt-1/2)}\|f\|_2  .
$$
This bound and Lemma \ref{NLL1} with $q=2$ imply the following lower bound.

\begin{Proposition}\label{RP1} For $\bt\in (0,1]$ and $r>0$ we have for $2\le p \le \infty$
$$
\varrho_m^o(\bA^r_\bt(\Tr^d),L_p) \gg m^{1-1/p-1/\bt-r/d}.
$$
\end{Proposition}

\section{Discussion}
\label{Di}

In this paper we emphasize importance of classes $\bA^r_\bt(\Psi)$, which are defined by structural conditions rather than by smoothness conditions. We give a brief history of the development of this idea. The first result, which connected the best approximations $\sigma_v(f,\Psi)_2$ with convergence of the series $\sum_{k}|\<f,\psi_k\>|$ was obtained by 
S.B. Stechkin \cite{Stech} in 1955 in the case of an orthonormal system $\Psi$. We formulate his result and a more general result momentarily. The following classes were introduced in \cite{DeTe} in studying sparse approximation with respect to arbitrary systems (dictionaries) $\D$ in a Hilbert space $H$. For a general dictionary $\D$ and for $\bt>0$ we 
define the class of functions (elements) 
$$
\cA^o_\bt(\Di,M) := \{f\in H\,:\, f=\sum_{k\in \Lambda} c_kg_k,\, g_k\in \D,\, |\Lambda| <\infty,\, 
\sum_{k\in \Lambda} |c_k|^\bt \le M^\bt\},
$$
and we define $\cA_\bt(\Di,M)$ as the closure (in $H$) of $\cA^o_\bt(\Di,M)$. Furthermore, we define $\cA_\bt(\Di)$ as the union of the classes $\cA_\bt(\Di,M)$ over all $M>0$. In the case of $\D$ being an orthonormal system S.B. Stechkin \cite{Stech} proved (see a discussion of this result in \cite{DTU}, Section 7.4)
\be\label{Di1}
f\in \cA_1(\D) \quad \text{if and only if}\quad \sum_{v=1}^\infty (v^{1/2}\sigma_v(f,\D)_2)\frac{1}{v} <\infty. 
\ee
A version of (\ref{Di1}) for the classes $\cA_\bt(\Di)$, $\bt\in (0,2)$, was obtained in \cite{DeTe}
\be\label{Di2}
f\in \cA_\bt(\D) \quad \text{if and only if}\quad \sum_{v=1}^\infty (v^{\alpha}\sigma_v(f,\D)_H)^\bt\frac{1}{v} <\infty, 
\ee
where $\al := 1/\bt-1/2$. In particular, (\ref{Di2}) implies that for $f\in \cA_\bt(\Di)$, $\bt\in (0,2)$, we have
\be\label{Di3}
\sigma_v(f,\D)_H \ll v^{1/2-1/\bt}.
\ee 
 
 We recall that relations (\ref{Di1}) -- (\ref{Di3}) were proved for an orthonormal system $\D$. 
 It is very interesting to note that (\ref{Di3}) actually holds for a general system $\D$ provided 
 $\bt \in (0,1]$. In the case $\bt =1$ it was proved by B. Maurey (see \cite{Pi}). For $\bt \in (0,1]$ it was proved in \cite{DeTe}. We note that classes $\cA_1(\Di)$ and their generalizations for the case of Banach spaces play a fundamental role in studying best $v$-term approximation and convergence properties of greedy algorithms with respect to general dictionaries $\D$ (see \cite{VTbook}). This fact encouraged experts to introduce and study function classes defined by restrictions on the coefficients of the functions' expansions. We explain this approach in detail. 
 
 Let as above a system $\Psi:=\{\psi_\bk\}_{\bk\in \Z^d}$ be indexed by $\bk\in\Z^d$. Consider a sequence of subsets $\cG:=\{G_j\}_{j=1}^\infty$, $G_j \subset \Z^d$, $j=1,2,\dots$, such that
 \be\label{Di4}
 G_1\subset G_2\subset \cdots \subset G_j\subset G_{j+1} \subset \cdots,\qquad \bigcup_{j=1}^\infty G_j =\Z^d.
 \ee
Consider functions representable in the form of absolutely convergent series 
\be\label{Direpr}
f = \sum_{\bk\in\Z^d} a_\bk(f)\psi_\bk,\qquad \sum_{\bk\in\Z^d} |a_\bk(f)|<\infty.
\ee
For $\bt \in (0,1]$ and $r>0$ consider the following class $\bA^r_\bt(\Psi,\cG)$ of functions $f$, which have representations (\ref{Direpr}) satisfying   conditions
\be\label{DiAr}
  \left(\sum_{  \bk\in G_j\setminus G_{j-1}} |a_\bk(f)|^\bt\right)^{1/\bt} \le 2^{-rj},\quad j=1,2,\dots,\quad G_0 :=\emptyset  .
\ee
One can also consider the following narrower version $\bA^r_\bt(\Psi,\cG,\infty)$ -- the class of functions $f$, which have representations (\ref{Direpr}) satisfying the  condition
\be\label{Di7}
\sum_{j=1}^\infty 2^{r\bt j} \sum_{  \bk\in G_j\setminus G_{j-1}} |a_\bk(f)|^\bt \le 1, \quad G_0 :=\emptyset  .
\ee

Probably, the first realization of the idea of the classes $\bA^r_\bt(\Psi,\cG)$ was done in \cite{VT150} in the special case, when $\Psi$ is the trigonometric system $\Tr^d := \{e^{i(\bk,\bx)}\}_{\bk\in\Z^d}$. We now proceed to the definition of classes $\bW^{a,b}_A(\Tr^d)$ from \cite{VT150}, which corresponds to the case $\bt=1$. Introduce the necessary notations.  
Let $\mathbf s=(s_1,\dots,s_d )$ be a  vector  whose  coordinates  are
nonnegative integers
$$
\rho(\mathbf s) := \bigl\{ \mathbf k\in\mathbb Z^d:[ 2^{s_j-1}] \le
|k_j| < 2^{s_j},\qquad j=1,\dots,d \bigr\},
$$
where $[a]$ means the integer part of a real number $a$.
For $f\in L_1 (\T^d)$
$$
\delta_{\mathbf s} (f,\mathbf x) :=\sum_{\mathbf k\in\rho(\mathbf s)}
\hat f(\mathbf k)e^{i(\mathbf k,\mathbf x)},\quad \hat f(\mathbf k) := (2\pi)^{-d}\int_{\T^d} f(\bx)e^{-i(\mathbf k,\mathbf x)}d\bx.
$$
Consider functions with absolutely convergent Fourier series. For such functions define the Wiener norm (the $A$-norm or the $\ell_1$-norm)
$$
\|f\|_A := \sum_{\mathbf k}|\hat f({\mathbf k})|.
$$
The following classes, which are convenient in studying sparse approximation with respect to the trigonometric system,
were introduced and studied in \cite{VT150}. Define for $f\in L_1(\T^d)$
$$
f_j:=\sum_{\|\bs\|_1=j}\delta_\bs(f), \quad j\in \N_0,\quad \N_0:=\N\cup \{0\}.
$$
For parameters $ a\in \R_+$, $ b\in \R$ define the class
$$
\bW^{a,b}_A:=\{f: \|f_j\|_A \le 2^{-aj}(\bar j)^{(d-1)b},\quad \bar j:=\max(j,1), \quad j\in \N\}.
$$
In this case 
$$
G_j := \bigcup_{\bs: \|\bs\|_1 \le j} \rho(\bs), \quad j=1,2,\dots,
$$
and classes $\bW^{a,b}_A$ with $b=0$ are similar to the above defined classes $\bA^a_1(\Tr^d,\cG)$. A more narrow version $\bA^a_1(\Tr^d,\cG,\infty)$ of these classes 
was studied recently in \cite{JUV}. 

The classes $\bA^r_\bt(\Psi)$ studied in this paper correspond to the case of $\bA^r_\bt(\Psi,\cG)$ with
\be\label{Di8}
G_j:= \{\bk\in \Z^d\,:\, \|\bk\|_\infty <2^j\},\quad j=1,2,\dots.
\ee
Note that classes $\bA^r_\bt(\Tr^d)$ are related to the periodic isotropic Nikol'skii classes $H^r_q$. There are several equivalent definitions of classes $H^r_q$. We give a convenient for us definition. Let $r>0$ and $1\le q\le \infty$. Class $H^r_q$ consists of periodic functions $f$ of $d$ variables satisfying the conditions
$$
\left\|\sum_{[2^{j-1}] \le \|\bk\|_\infty< 2^j} \hat f(\bk) e^{i(\bk,\bx)}\right\|_q \le 2^{-rj}, \quad j=0,1,\dots.
$$
For instance, it is easy to see that in the case $q=2$ and $r/d >1/2$ the class $H^r_q$ is embedded in the class $\bA^{r-d/2}_1(\Tr^d)$. However, the class $\bA^{r-d/2}_1(\Tr^d)$ is 
substantially larger than $H^r_2$. For instance, take $\bk \in G_j\setminus G_{j-1}$ with $G_j$ defined in (\ref{Di8}). Then the function $g_\bk := 2^{-(r-d/2)j}e^{i(\bk,\bx)}$ belongs to $\bA^{r-d/2}_1(\Tr^d)$ but does not belong to any $H^{r'}_2$ with $r'> r-d/2$. 

We now give a brief comparison of sampling recovery results for classes $H^r_q$ and $\bA^r_\bt(\Tr^d)$. Recall the setting of the optimal  linear recovery. For a fixed $m$ and a set of points  $\xi:=\{\xi^j\}_{j=1}^m\subset \Omega$, let $\Phi $ be a linear operator from $\bbC^m$ into $L_p(\Omega,\mu)$.
Denote for a class $\bF$ (usually, centrally symmetric and compact subset of $L_p(\Omega,\mu)$) (see \cite{VT51})
$$
\varrho_m(\bF,L_p) := \inf_{\text{linear}\, \Phi; \,\xi} \sup_{f\in \bF} \|f-\Phi(f(\xi^1),\dots,f(\xi^m))\|_p.
$$

It is known (see \cite{VTbookMA}, Theorem 3.6.4, p.125) that for all $1\le p,q\le \infty$, $r >d/q$, we have 
\be\label{Di9}
\varrho_m(H^r_q,L_p) \asymp m^{-r/d +(1/q-1/p)_+},\quad (a)_+ := \max(a,0).
\ee 
Clearly, (\ref{Di9}) implies the same upper bound for the $\varrho^o_m(H^r_q,L_p)$. 
Results on the lower bounds for $\varrho^o_m(\Tr(\bN,d)_q,L_p)$ from \cite{VT202} (see Lemma 4.3 there) and Lemma \ref{NLL1} of this paper show that the following relation holds
\be\label{Di10}
\varrho^o_m(H^r_q,L_p) \asymp m^{-r/d +(1/q-1/p)_+}.
\ee
In particular,
\be\label{Di11}
\varrho^o_m(H^r_2,L_p) \asymp m^{-r/d +(1/2-1/p)_+}.
\ee
The lower bound (1.5) from \cite{VT202} and the lower bound in (\ref{I13}) of this paper give the following bound in the case $\bt=1$
\be\label{Di12}
\varrho^o_m(\bA^r_1(\Tr^d,L_p) \gg m^{-1/2 +(1/2-1/p)_+ -r/d}.
\ee
The upper bound in (\ref{I13}) of this paper gives the following bound 
\be\label{Di13}
\varrho^o_m(\bA^r_1(\Tr^d,L_p) \ll \left(\frac{m}{(\log m)^3}\right)^{-1/2 +(1/2-1/p)_+ -r/d}.
\ee
Relations (\ref{Di11})--(\ref{Di13}) mean that we obtain close bounds for the class $H^r_2$ 
and for the larger class $\bA^{r-d/2}_1(\Tr^d)$, $r>d/2$. Note that for the class $H^r_2$ we 
obtain the same bounds for the linear sampling recovery. It is proved in \cite{VT202} that for 
the linear sampling recovery in $\bA^r_\bt(\Tr^d)$ we have
\be\label{Di14}
\varrho_m(\bA^r_\bt(\Tr^d,L_2) \gg m^{-r/d}.
\ee
Inequalities (\ref{Di14}) with $\bt=1$ and (\ref{Di13}) with $p=2$ demonstrate that nonlinear 
sampling recovery provides better error guarantees than linear sampling recovery.

  \Addresses


\begin{thebibliography}{9999}
 
 \bibitem{Bour} J. Bourgain, An improved estimate in the restricted isometry problem, In Geometric Aspects of Functional Analysis, volume 2116 of Lecture Notes in Mathematics, pages 65--70. Springer, 2014.
 
 



  \bibitem{DPTT} F. Dai, A. Prymak, V.N. Temlyakov, and  S.U. Tikhonov,  
  Integral norm discretization and related problems,  (Russian) {\it Uspekhi Mat. Nauk  }{\bf 74} (2019), no. 4(448), 3–58; translation in {\it Russian Math. Surveys} {\bf  74 } (2019), no. 4, 579--630 .
 
\bibitem{DT} F. Dai and V.N.  Temlyakov, Universal sampling discretization, {\it Constr. Approx.} (2023). Published: 25-04-2023. 
https://doi.org/10.1007/s00365-023-09644-2.

  \bibitem{DTM1} F.  Dai and V.N.  Temlyakov, Universal discretization and sparse sampling recovery,
 arXiv:2301.05962v1 [math.NA] 14 Jan 2023.
 
 \bibitem{DTM2} F. Dai and V.N. Temlyakov, Random points are good for universal discretization, J. Math. Anal. Appl., {\bf 529} (2024) 127570; arXiv.2301.12536[math.FA] 5 Feb 2023.
 
\bibitem{DTM3} F. Dai and V.N. Temlyakov, Lebesgue-type inequalities in sparse sampling recovery,
arXiv:2307.04161v1 [math.NA] 9 Jul 2023.

\bibitem{DeTe} R.A. DeVore and V.N. Temlyakov, Some remarks on greedy algorithms, 
Advances on Comput. Math., {\bf 5} (1996), 173--187. 

 \bibitem{DTU} Dinh D{\~u}ng, V.N. Temlyakov, and T. Ullrich, Hyperbolic Cross Approximation, Advanced Courses in Mathematics CRM Barcelona, Birkh{\"a}user, 2018; arXiv:1601.03978v2 [math.NA] 2 Dec 2016.
  
 
  
    \bibitem{HR} I. Haviv and O. Regev, The restricted isometry property of subsampled Fourier matrices, In Geometric aspects of functional analysis, volume 2169 of Lecture Notes in Math., pages 163--179. Springer, Cham, 2017.
 
 \bibitem{JUV} T. Jahn, T. Ullrich, and F. Voigtlaender, Sampling numbers of smoothness classes via 
 $\ell^1$-minimization, arXiv:2212.00445v1 [math.NA] 1 Dec 2022. 
 
 \bibitem{KKLT} B.S.  Kashin, E. Kosov, I. Limonova, and V.N.  Temlyakov, Sampling discretization and related problems,
 {\it J. Complexity} {\bf  71} (2022), Paper No. 101653.
 
 
 

 


  \bibitem{LT} E. Livshitz and V. Temlyakov,  Sparse approximation and recovery by greedy algorithms, IEEE Transactions on Information Theory, {\bf 60} (2014), 3989--4000;
arXiv: 1303.3595v1 [math.NA] 14 Mar 2013.
 
 


\bibitem{Pi} G. Pisier, Remarques sur un resultat non publi{\'e} de B. Maurey, Seminaire 
d'Analyse Fonctionalle, 1980-1981, Ecole Polytechnique Centre de Mathematiques, Palaiseau. 

\bibitem{Stech} S.B. Stechkin, On absolute convergence of orthogonal series, Dokl. AN SSSR, {\bf 102} (1955), 37--44 (in Russian).



\bibitem{VT51} V.N. Temlyakov, On Approximate Recovery of Functions with Bounded Mixed Derivative, J. Complexity, {\bf 9} (1993), 41--59.
 
   \bibitem{T1} V.N. Temlyakov,  Greedy algorithms in Banach spaces,
{\it Adv. Comput. Math.}  \textbf{14}  (2001), 277--292.
 
  \bibitem{VTbook} V.N. Temlyakov, Greedy Approximation, Cambridge University
Press, 2011.
 
  
 \bibitem{VT150} V.N. Temlyakov, Constructive sparse trigonometric approximation and other problems for functions with mixed smoothness, arXiv: 1412.8647v1 [math.NA] 24 Dec 2014, 1--37; Matem. Sb., {\bf 206} (2015), 131--160. 
 
   
 
 \bibitem{VTbookMA} V. Temlyakov, {\em Multivariate Approximation}, Cambridge University Press, 2018.


\bibitem{VT202} V. Temlyakov, Sparse sampling recovery by greedy algorithms, arXiv:2312.13163v2 [math.NA] 30 Dec 2023.
  

\bibitem{TWW} J.F. Traub, G.W. Wasilkowski, and H. Wo{\'z}niakowski, Information-Based Complexity, Academic Press, Inc., 1988.


 
  \end{thebibliography}
\end{document}